\documentclass[11pt,a4paper,twoside]{article}
\usepackage{bm, amsmath, amssymb, amsthm} %, amsthm

\def\<{\langle}
\def\>{\rangle}
\def\RR{\mathbb{R}}
\def\NN{\mathbb{N}}

\newcommand\tr{\operatorname{Tr}}
\newcommand\Div{\operatorname{div}}
\newcommand\id{\operatorname{id}}
\def\Ric{\operatorname{Ric}}
\def\vol{\operatorname{vol}}

\def\eq{\hspace*{-1.5mm}&=&\hspace*{-1.2mm}}
\def\plus{\hspace*{-1.5mm}&+&\hspace*{-1.2mm}}

\newtheorem{corollary}{Corollary}
\newtheorem{definition}{Definition}
\newtheorem{example}{Example}
\newtheorem{remark}{Remark}
\newtheorem{lemma}{Lemma}
\newtheorem{proposition}{Proposition}
\newtheorem{theorem}{Theorem}

\author{Vladimir Rovenski\footnote{Mathematical Department, University of Haifa, Mount Carmel, 31905 Haifa,  Israel
       \newline e-mail: {\tt vrovenski@univ.haifa.ac.il} } }

\title{Integral formulas for a Riemannian manifold with several orthogonal complementary distributions}

\begin{document}

\date{}

\maketitle

\begin{abstract}
In~the paper we prove integral formulae for a Riemannian manifold endowed with $k>2$ orthogonal complementary  distributions,
which generalize well-known formula for $k=2$,
and give applications to splitting and isometric immersions of Riemannian manifolds, in particu\-lar, multiply warped products, and to hypersurfaces with $k>2$ distinct principal curvatures of constant multiplicities.

\vskip1.5mm\noindent
\textbf{Keywords}:
Almost $k$-product manifold, distribution, foliation,
mean curvature vector, mixed scalar curvature,
isometric immersion, multiply warped product,
Dupin hypersurface

\vskip1.5mm
\noindent
\textbf{Mathematics Subject Classifications (2010)} 53C15; 53C12; 53C40

\end{abstract}

\section*{Introduction}

Distributions on a manifold (i.e., subbundles of the tangent bundle) are used to build up notions of integrability, and specifically of a foliated manifold.
Distributions and foliations on Riemannian manifolds appear in various situations, e.g., \cite{bf,g1967,rov-m},
among them warped products are fruitful generalizations of the direct product playing important role in 
mathematics and physics, 
%differential geometry as well as in mathematical physics, 
e.g., \cite{chen1}.

 Integral formulae are useful for many problems in the theory of foliations, e.g. \cite{arw2014,rov-m,wa1}:

a)~the existence and characterizing of foliations, whose leaves have a given geometric property;
%  such as being totally geodesic, totally umbilical or minimal;

b)~prescribing the {higher mean curvatures} of the leaves of a foliation;

c)~minimizing functionals like volume and energy defined for tensor fields on a manifold.

\noindent
The~first known integral formula for a closed Riemannian manifold
endowed with a codimension one foliation tells us that the integral mean curvature
of the leaves vanishes, see~\cite{reeb1}. The~second formula in the series of total $\sigma_k$'s --
elementary symmetric functions of principal curvatures of the leaves -- says that for a codimension one foliation with a unit normal $N$ to the leaves the total $\sigma_2$ is a half of the {total Ricci curvature} in the $N$-direction,~e.g., \cite{arw2014}:
\begin{equation}\label{E-sigma2}
 \int_M (2\,\sigma_2-\Ric_{N,N})\,{\rm d}\vol=0,
\end{equation}
which is a consequence of Stokes' theorem applied to $\nabla_N\,N+\sigma_1 N$.
One can see directly that \eqref{E-sigma2} implies nonexistence of totally umbilical foliations on a closed manifold of negative curvature: if the integrand is strictly positive, so is the value of the integral.
The mixed scalar curvature ${\rm S}_{\rm mix}$
%(being an averaged mixed sectional curvature)
is one of the simplest curvature invariants of a pseudo-Riemannian almost-product structure, which can be defined as an averaged sum of sectional curvatures of planes that non-trivially intersect with both of the distributions.
The~following integral formula for a closed Riemannian manifold $(M,g)$ endowed with two complementary orthogonal distributions,
 %${\cal D}_1$ and ${\cal D}_2$, 
 see \cite{wa1} (and \cite{ra1} for foliations),
 generalizes \eqref{E-sigma2} and has many interesting global corollaries
 (e.g., decomposition criteria using the sign of the mixed scalar curvature, \cite{step1}):
\begin{eqnarray}\label{E-PW2-int}
 \int_M\big({\rm S}_{\rm mix}
 %({\cal D}_1,{\cal D}_2)
 +\|h_1\|^2+\|h_2\|^2-\|H_1\|^2-\|H_2\|^2-\|T_1\|^2-\|T_2\|^2\big)\,{\rm d}\vol_g=0.
\end{eqnarray}
Here 
%${\rm S}_{\rm mix}({\cal D}_1,{\cal D}_2)$ is the \textit{mixed scalar curvature} of $(M,g;{\cal D}_1,{\cal D}_2)$,
and $h_i,H_i=\tr_g h_i$ and $T_i$ are the second fundamental form, the mean curvature vector field and the integrability tensor of distributions ${\cal D}_1$ and ${\cal D}_2$ on $(M,g)$.
The~formula \eqref{E-PW2-int} was obtained by calculation of the divergence of the vector field $H_1+H_2$,
\begin{equation}\label{E-PW}
 \Div(H_1 + H_2) = {\rm S}_{\rm mix}
 %({\cal D}_1,{\cal D}_2) 
 +\|h_1\|^2+\|h_2\|^2-\|H_1\|^2-\|H_2\|^2-\|T_1\|^2-\|T_2\|^2,
\end{equation}
 and then applying Stokes' theorem.

Recently, the notion of warped products was naturally extended to multiply warped products, e.g., \cite{chen1}.
The notion of {multiply warped product}, in turn, is a special case of a Riemannian almost $k$-product structure with $k\in\{2,\ldots, n\}$ foliations.
This structure can be also viewed in the theory of webs composed of foliations of different dimensions, see~\cite{AG2000}.
A Riemannian almost $k$-product structure appears also on a proper Dupin hypersurface of a real space-form,
i.e., the number $k$ of distinct principal curvatures is constant and each
principal curvature function is constant along its corresponding surface of curvature, see~\cite{cecil-ryan}.

Thus, the problem of generalizing \eqref{E-PW2-int} to the case of $(M,g)$ with $k>2$ distributions is actual.
In Section~\ref{sec:if}, we solve this problem for arbitrary $k>2$.
In Sections~\ref{sec:app} and \ref{sec:dupin} we give applications to splitting and isometric immersions of manifolds, in particular, multiply warped products, and to hypersurfaces with $k>2$ distinct principal curvatures of constant multiplicities.

\section{New integral formulas}
\label{sec:if}

 Let an $n$-dimensional Riemannian manifold $(M,g)$ with the Levi-Civita connection $\nabla$ and the curvature tensor $R$ be endowed with $k>2$ orthogonal $n_i$-dimensional distributions ${\cal D}_i\ (1\le i\le k)$ with $\sum n_i =\dim M$.
There exists on $M$ a~local adapted orthonormal frame $\{E_1,\ldots,E_n\}$, where
\[
 \{E_1,\ldots, E_{n_1}\}\subset{{\cal D}_1},\quad
 \{E_{n_{i-1}+1},\ldots, E_{n_i}\}\subset{{\cal D}_i},\quad 2\le i\le k.
\]
Such $(M,g;{\cal D}_1,\ldots,{\cal D}_k)$ is called a \textit{Riemannian almost $k$-product manifold}, see \cite{g1967} for $k=2$.
A~plane in $TM$ spanned by two vectors belonging to different distributions, say, ${\cal D}_i$ and~${\cal D}_j$
is called~\textit{mixed}.
We will generalize the well-known concept of almost-product manifolds.

\begin{definition}\rm
The function on $(M,g)$ with $k\ge2$ non-degenerate distributions,
\begin{equation}\label{E-Smix-k}
 {\rm S}_{\,\rm mix}({\cal D}_1,\ldots,{\cal D}_k)=\sum\nolimits_{\,i<j}{\rm K}({\cal D}_i,{\cal D}_j)
\end{equation}
will be called the \textit{mixed scalar curvature} of
$(M,g;{\cal D}_1,\ldots,{\cal D}_k)$, where
\[
 {\rm K}({\cal D}_i,{\cal D}_j) = \sum\nolimits_{\,n_{i-1}<a\,\le n_i,\ n_{j-1}<b\le n_j}
 \<R(E_a,{E}_b)\,E_a,\,{E}_{b}\>,\quad i\ne j.
\]
\end{definition}

Let $P_i:TM\to{\cal D}_i$ be the orthoprojector, and $\widehat P_{i}=\id_{TM}-P_i$ be the orthoprojector onto ${\cal D}_{i}^\bot$.
The~second fundamental form $h_i:{\cal D}_i\times {\cal D}_i\to {\cal D}_i^\bot$ (symmetric)
and the integrability tensor $T_i:{\cal D}_i\times {\cal D}_i\to {\cal D}_i^\bot$ (skew-symmetric)
of ${\cal D}_i$ are defined by
\[
 h_i(X,Y) = \frac12\,\widehat P_{i}(\nabla_XY+\nabla_YX),\quad
 T_i(X,Y) = \frac12\,\widehat P_{i}(\nabla_XY-\nabla_YX)=\frac12\,\widehat P_{i}\,[X,Y].
\]
Let $h_{ij},\,H_{ij}=\tr_g h_{ij},\,T_{ij}$ be
the~second fundamental forms, mean curvature vector fiels and the integrability tensors
of distributions ${\cal D}_{i,j}={\cal D}_i\oplus{\cal D}_j$ in $M$,
 and $P_{ij}:TM\to{\cal D}_{i,j}$ be orthoprojectors, etc. Note that
\[
 H_i=\sum\nolimits_{\,j\ne i} P_j H_i,\quad
 H_{1\ldots r}=P_{r+1\ldots k}(H_1+\ldots+H_r),\quad {\rm etc}.
\]
Recall that a distribution ${\cal D}_i$ is called integrable if $T_i=0$,
and ${\cal D}_i$ is called {totally umbilical}, {harmonic}, or {totally geodesic},
if ${h}_i=({H}_i/n_i)\,g,\ {H}_i =0$, or ${h}_i=0$, respectively.

\smallskip

Let $S(r,k)$ be the set of all $r$-combinations (i.e., subsets of $r$ distinct elements) of $\{1,\ldots, k\}$.
The~$S(r,k)$ contains $C^r_k=\frac{k!}{r!(k-r)!}$ elements (where $C^r_k$ denote binomial coefficients).
For example, $S(k-1,k)$ contains $k$ elements.
For any ${\bm q}\,\in S(r,k)$ we may assume $q_1<\ldots<q_r$. Set $h_{\bm q}=h_{q_1,\ldots,q_r}$ and $T_{\bm q}=T_{q_1,\ldots,q_r}$.

Our main goal is the following formula for $k\in\{2,\ldots, n\}$, which generalizes \eqref{E-PW}.

\begin{theorem}\label{T-k3} For a Riemannian almost $k$-product manifold $(M,g;{\cal D}_1,\ldots,{\cal D}_k)$ we have
\begin{eqnarray}\label{E-PW3-k}
 \Div X = 2\,{\rm S}_{\rm mix}({\cal D}_1,\ldots,{\cal D}_k)
 +\sum\nolimits_{\,r\in\{1,k-1\}}\sum\nolimits_{\,{\bm q}\,\in S(r,k)} (\|h_{\bm q}\|^2-\|H_{\bm q}\|^2-\|T_{\bm q}\|^2),
\end{eqnarray}
where $X=\sum\nolimits_{\,r\in\{1,k-1\}}\sum\nolimits_{\,{\bm q}\,\in S(r,k)} H_{\bm q}$.
\end{theorem}

\begin{proof}
For $k=2$ we have \eqref{E-PW}.
To illustrate the proof for $k>2$, first consider the case of $k=3$.
Using \eqref{E-PW} for the distributions ${\cal D}_1$ and ${\cal D}_{1}^\bot={\cal D}_2\oplus{\cal D}_3$, we get
\begin{equation}\label{E-PW3-1}
 \Div(H_1 + H_{23}) = 2\,{\rm S}_{\rm mix}({\cal D}_1, {\cal D}_{1}^\bot)
 + (\|h_1\|^2 - \|H_1\|^2 - \|T_1\|^2) + (\|h_{23}\|^2 -\|H_{23}\|^2 - \|T_{23}\|^2),
\end{equation}
and similarly for $({\cal D}_2,\,{\cal D}_{2}^\bot)$ and $({\cal D}_3,\,{\cal D}_{3}^\bot)$.
Summing 3 copies of \eqref{E-PW3-1}, we obtain \eqref{E-PW3-k} for $k=3$:
\begin{eqnarray}\label{E-PW-3}
\nonumber
 && \Div\big(\sum\nolimits_{\,i} H_i +\sum\nolimits_{\,i<j} H_{ij}\big) = 2\,{\rm S}_{\rm mix}({\cal D}_1,{\cal D}_2,{\cal D}_3)  \\
 && +\sum\nolimits_{\,i}(\|h_i\|^2 -\|H_i\|^2 - \|T_i\|^2) +\sum\nolimits_{\,i<j}(\|h_{ij}\|^2 - \|H_{ij}\|^2 -\|T_{ij}\|^2) .
\end{eqnarray}
Next, for arbitrary~$k>2$, we apply \eqref{E-PW} for pairs of distributions $({\cal D}_i,\,{\cal D}_{i}^\bot)$ and get $k$ equalities.
Summing these equations and using Lemma~\ref{L-smix} below, we get \eqref{E-PW3-k}.
\end{proof}

\begin{lemma}\label{L-smix} We have
\begin{equation*}
%\label{E-Dk-Smix}
 2\,{\rm S}_{\,\rm mix}({\cal D}_1,\ldots,{\cal D}_k) = \sum\nolimits_{\,i}{\rm S}_{\,\rm mix}({\cal D}_i,{\cal D}^\bot_i).
\end{equation*}
\end{lemma}

\begin{proof}
This directly follows from definitions \eqref{E-Smix-k}.
\end{proof}

\begin{remark}\rm
Just the particular case of \eqref{E-PW3-k} for $k=3$, see \eqref{E-PW-3}, can find many geometrical applications,
because an~almost 3-product structure appears naturally in several topics:

\noindent\
1) almost para-$f$-manifolds, e.g., \cite{tar}.

\noindent\
2) lightlike manifolds, i.e., $(M,g)$ with degenerate metric $g$ of constant rank and index, see \cite{dug}.

\noindent\
3) on orientable 3-manifolds, since they admits 3 linearly independent vector~fields, i.e., $n_i=1$.

\noindent\
4) the theory of 3-webs composed of three generic foliations of different dimensions, e.g.,~\cite{tol}.

\noindent\
5) minimal hypersurfaces in the sphere with 3 distinct principal curvatures, see \cite{ots-86}.

\noindent\
6) tubes over a standard embeddings of a projective plane $FP^2$, for $F = \RR; \mathbb{C};\mathbb{H}$ or $\mathbb{O}$ (Cayley numbers), in $S^4; S^7; S^{13}$, and $S^{25}$, respectively, see~\cite{cecil-ryan}.
\end{remark}

\begin{example}
%\label{Ex-k-345}
\rm
For $k=4$ (next to $k=2,3$), formula \eqref{E-PW3-k} reads~as
\begin{eqnarray*}
%%%%% k=4
%\label{E-PW-4}
\nonumber
 && \Div\big(\sum\nolimits_{i} \!H_i +\sum\nolimits_{i<j<s} \!H_{ijs}\big)
 = 2\,{\rm S}_{\rm mix}({\cal D}_1,\ldots,{\cal D}_4)  \\
 && +\sum\nolimits_{i}(\|h_i\|^2 -\|H_i\|^2 - \|T_i\|^2) +\sum\nolimits_{\,i<j<s}(\|h_{ijs}\|^2 -\|H_{ijs}\|^2 - \|T_{ijs}\|^2) ,\\
%%%%%%%%% k=5
\end{eqnarray*}
\end{example}

Using Stokes' theorem for \eqref{E-PW3-k} on a closed manifold $(M,g;{\cal D}_1,\ldots,{\cal D}_k)$ yields the novel \textit{integral formula} for any $k\in\{2,\ldots,n\}$, which for $k=2$ coincides with \eqref{E-PW2-int}.

\begin{corollary}
On a closed Riemannian almost $k$-product manifold we have
\begin{equation}\label{E-int-k}
 \int_M\Big( 2\,{\rm S}_{\rm mix}({\cal D}_1,\ldots,{\cal D}_k)
 +\sum\nolimits_{\,r\in\{1,k-1\}}\sum\nolimits_{\,{\bm q}\,\in S(r,k)}(\|h_{\bm q}\|^2-\|H_{\bm q}\|^2-\|T_{\bm q}\|^2)
 \Big)\,{\rm d}\vol_g =0.
\end{equation}
\end{corollary}

\begin{example}
%\label{Ex-k-345b}
\rm
For few initial values of $k$, $k=3,4$, the integral formula \eqref{E-int-k} reads as follows:
\begin{eqnarray*}
%\label{E-PW3-3-4-5}
%\nonumber
 &&\hskip-6mm \int_M\!\Big(2\,{\rm S}_{\rm mix}({\cal D}_1,{\cal D}_2,{\cal D}_3)
  {+}\!\sum\nolimits_{i}\!\big(\|h_i\|^2 {-}\|H_i\|^2 {-}\|T_i\|^2\big)
  {+}\!\sum\nolimits_{i<j}\!\big(\|h_{ij}\|^2 {-}\|H_{ij}\|^2 {-}\|T_{ij}\|^2\big)\Big){\rm d}\vol_g =0,\\
%%%%% k=4
 &&\hskip-8mm \int_M\!\!\big( 2\,{\rm S}_{\rm mix}({\cal D}_1,...,{\cal D}_4)
 {+}\!\!\sum\nolimits_{i}(\|h_i\|^2 {-}\|H_i\|^2 {-}\|T_i\|^2)
 % \\ &&
  {+}\!\!\sum\nolimits_{i<j<s}(\|h_{ijs}\|^2 {-}\|H_{ijs}\|^2 {-} \|T_{ijs}\|^2)\big){\rm d}\vol_g \!=0 ,\\
\end{eqnarray*}
\end{example}

In order to simplify the LHS of \eqref{E-PW3-k} to the shorter view $\Div(\sum\nolimits_{\,i} H_i)$, we reorganize the terms
$\Div\big(\sum\nolimits_{\,{\bm q}\,\in S(k-1,k)} H_{\bm q} \big)$ for $r>1$ and obtain new (auxiliary) integral formulae.

\begin{theorem}\label{T-div-k}
For $(M,g;{\cal D}_1,\ldots,{\cal D}_k)$ and any $r\in\{2,\ldots, k-1\}$, we have
\begin{equation}\label{E-L-k4}
%\nonumber
 \Div X
 %\big(\sum\nolimits_{\,{\bm q}\,\in S(r,k)} H_{\bm q} - C^{\,r-1}_{k-2}\sum\nolimits_{\,i} H_i \big) 
 %\\ && 
 = \sum\nolimits_{\,{\bm q}\,\in S(r,k)} \big(\|H_{\bm q}\|^2
 +\<\,\sum\nolimits_{\,i=1}^r H_{q(i)} - r\,H_{\bm q}\,,\sum\nolimits_{\,j\not\in{\bm q}} H_{j}\> \big),
\end{equation}
where $X=\sum\nolimits_{\,{\bm q}\,\in S(r,k)} H_{\bm q} - C^{\,r-1}_{k-2}\sum\nolimits_{\,i} H_i$.
\end{theorem}

\begin{proof}
Using equality $H_{1\ldots r}=P_{r+1\ldots k}(H_1+\ldots+H_r)$ for ${\bm q}=\{1,\ldots, r\}$, we find
\begin{eqnarray*}
 \Div H_{1\ldots r} \eq \Div_{r+1\ldots k} H_{1\ldots r} -\|H_{1\ldots r}\|^2 \\  %% , H_1+H_2\> \\
 \eq \Div_{r+1\ldots k}(H_{1}+\ldots+H_r) +\<H_1+\ldots+H_r,\ H_{r+1\ldots k}\> - \|H_{1\ldots r}\|^2, %% - 2\<H_1,H_2\> ,
\end{eqnarray*}
and similarly for all $C^r_k$ cases of ${\bm q}\in S(r,k)$. Summing the above, we use equalities
\[
 \Div_{r+1\ldots k} H_{1}=\sum\nolimits_{\,j>r} \Div_{j} H_{1},\quad
 \Div_{\,2\ldots k} H_1= \Div H_1 +\|H_1\|^2,\quad {\rm etc}.
\]
We apply \eqref{E-L-k4} with $k-1$ for the distributions ${\cal D}_1,\ldots,{\cal D}_{k-2}$ and ${\cal D}_{k-1,k}$,
and then obtain similar formulas for other choices of a pair of distributions. Summing these $C^2_k$ equations, we get equation of the form \eqref{E-L-k4}
with $k$, comparing coefficients yield the claim.
\end{proof}

\begin{corollary}
The following integral formulae for $r\in\{2,\ldots, k-1\}$ take place on a closed Riemannian almost $k$-product manifold:
\begin{equation}\label{E-k-r}
 \int_M\sum\nolimits_{\,{\bm q}\,\in S(r,k)} \big(\|H_{\bm q}\|^2
 +\<\sum\nolimits_{\,i=1}^r H_{q(i)} - r\,H_{\bm q}\,,\, \sum\nolimits_{\,j\not\in{\bm q}} H_{j}\>\big)\,{\rm d}\vol_g =0 .
\end{equation}
\end{corollary}

\begin{example}
%\label{Ex-HH-345}
\rm
For $k=3$ and $r=2$, the formula \eqref{E-k-r} reads as
%yields the following integral formula:
\begin{equation*}
%\label{E-HH-int3}
%%%%% k=3
 \int_M\big(\sum\nolimits_{\,i} \|H_i\|^2 + \sum\nolimits_{\,i<j} \big(2\<H_i,H_j\> -\|H_{ij}\|^2\big)\big)\,{\rm d}\vol_g =0,
\end{equation*}
and for $k=4$ and $r=2$, the formula \eqref{E-k-r} reads as
%yields the following integral formula:
\begin{equation*}
%\label{E-HH-int4-2}
%%%%% k=4, r=2
 \int_M\big( 2\sum\nolimits_{\,i} \|H_i\|^2 +\sum\nolimits_{\,i,\,j<s} \<H_i,\,H_{js}\> -\sum\nolimits_{\,i<j} \|H_{ij}\|^2
 \big)\,{\rm d}\vol_g =0.
\end{equation*}
The reader can easily find the integral formula corresponding to \eqref{E-k-r} with $k=4$ and $r=3$.
\end{example}

From Theorem~\ref{T-k3} and Theorem~\ref{T-div-k} for $r=k-1$, we obtain the following companion of \eqref{E-PW3-k}.

\begin{theorem}
%\label{T-k-gen}
For a Riemannian almost $k$-product manifold $(M,g;{\cal D}_1,\ldots,{\cal D}_k)$ we have
\begin{eqnarray*}
%\label{E-PW-k-gen}
 &&\hskip-6mm 2\Div\big(\sum\nolimits_{\,i} H_i \big) = 2\,{\rm S}_{\rm mix}({\cal D}_1,\ldots,{\cal D}_k)
 +\sum\nolimits_{\,r\in\{1,k-1\}}\sum\nolimits_{\,{\bm q}\,\in S(r,k)}\big(\|h_{\bm q}\|^2-\|H_{\bm q}\|^2-\|T_{\bm q}\|^2\big) \\
 &&
 -\sum\nolimits_{\,{\bm q}\,\in S(k-1,k)}\big(\|H_{\bm q}\|^2+\<\,\sum\nolimits_{\,i=1}^{k-1} H_{q(i)} - (k-1)\,H_{\bm q}\,,\sum\nolimits_{\,j\not\in{\bm q}} H_{j}\,\> \big) .
\end{eqnarray*}
\end{theorem}

\begin{corollary}\label{C-k-2} The following integral formula takes place on a closed $(M,g;{\cal D}_1,\ldots,{\cal D}_k)$:
\begin{eqnarray*}
 &&\hskip-6mm \int_M\Big( 2\,{\rm S}_{\rm mix}({\cal D}_1,\ldots,{\cal D}_k)
  +\sum\nolimits_{\,r\in\{1,k-1\}}\sum\nolimits_{\,{\bm q}\,\in S(k-1,k)}\big(\|h_{\bm q}\|^2-\|H_{\bm q}\|^2-\|T_{\bm q}\|^2\big)\\
 &&
 -\sum\nolimits_{\,{\bm q}\,\in S(k-1,k)}\big(\|H_{\bm q}\|^2+\<\,\sum\nolimits_{\,i=1}^{k-1} H_{q(i)} - (k-1)\,H_{\bm q}\,, \sum\nolimits_{\,j\not\in{\bm q}} H_{j}\,\> \big)\Big)\,{\rm d}\vol_g =0 .
\end{eqnarray*}
\end{corollary}

\begin{example}
%\label{C-k34-old}
\rm
For particular case of $k=3$, the integral formula of Corollary~\ref{C-k-2} reads as
\begin{eqnarray*}
%\label{E-PW3-int-k3}
%\nonumber
 && \int_M\Big({\rm S}_{\rm mix}({\cal D}_1,{\cal D}_2,{\cal D}_3) -\sum\nolimits_{\,i} \|H_i\|^2 -\sum\nolimits_{\,i<j} \<H_i,\,H_j\>
 + \frac12\sum\nolimits_{\,i} \big( \|h_i\|^2 - \|T_i\|^2\big) \\
 &&\qquad +\,\frac12\sum\nolimits_{\,i<j} \big( \|h_{ij}\|^2 - \|T_{ij}\|^2\big)\Big)\,{\rm d}\vol_g =0.
\end{eqnarray*}
The above formula was obtained in \cite{BM-1} by long direct calculations of $\Div(H_1+H_2+H_3)$.
\end{example}

\section{Splitting and isometric immersions of manifolds}
\label{sec:app}

Here, we use results of Section~\ref{sec:if} to prove some splitting and non-existence of
immersions results for Riemannian almost $k$-product manifolds with $k>2$.
 We say that $(M,g;{\cal D}_1,\ldots,{\cal D}_k)$ \textit{splits} if
all distributions  ${\cal D}_i$ are integrable and
$M$ is locally the direct product $M_1\times\ldots\times M_k$ with foliations tangent to ${\cal D}_i$.
Recall that if a simply connected manifold splits then it is the direct~product.

We apply the submanifolds theory to Riemannian almost $k$-product manifolds.

\begin{definition}\rm
A pair $({\cal D}_i,{\cal D}_j)$ with $i\ne j$ of distributions on
$(M,g;{\cal D}_1,\ldots,{\cal D}_k)$ with $k>2$ is called

a) \textit{mixed totally geodesic}, if $h_{ij}(X,Y)=0$ for all $X\in{\cal D}_i$ and $Y\in{\cal D}_j$.

b) \textit{mixed integrable}, if $T_{ij}(X,Y)=0$ for all $X\in{\cal D}_i$ and $Y\in{\cal D}_j$.
\end{definition}

\begin{lemma}\label{L-mixed-YU-TG}
If each pair $({\cal D}_i,{\cal D}_j)$ on $(M,g;{\cal D}_1,\ldots,{\cal D}_k)$ with $k>2$ is

a$)$ {mixed totally geodesic}, then $h_{\bm q}(X,Y)=0$ \ \
b$)$ {mixed integrable}, then $T_{\bm q}(X,Y)=0$
\newline
for all ${\bm q}\in S(r,k)$, $2<r<k$ and $X\in{\cal D}_{q(1)},\,Y\in{\cal D}_{q(2)}$.
\end{lemma}

\begin{proof} This follows by mathematical induction.
\end{proof}

The next our results generalize \cite[Theorem~6]{step1} and \cite[Theorem~2]{wa1} on case $k=2$.

\begin{theorem}
%\label{T-214}
Let a Riemannian almost $k$-product manifold $(M,g;{\cal D}_1,\ldots,{\cal D}_k)$ with $k>2$ has integrable harmonic distributions
${\cal D}_1,\ldots,{\cal D}_k$. If~$\,{\rm S}_{\rm mix}({\cal D}_1,\ldots,{\cal D}_k)\ge0$,
and each pair $({\cal D}_i,{\cal D}_j)$ is mixed integrable, then $M$ splits.
\end{theorem}

\begin{proof}
From the equality $H_{1\ldots r}=P_{r+1\ldots k}(H_1+\ldots+H_r)$ follows that
$H_i=0$ for all $i\in\{1,\ldots, k\}$, then $H_{\bm q}=0$ for all ${\bm q}\,\in S(r,k)$ and $2\le r<k$.
Similarly (by Lemma~\ref{L-mixed-YU-TG}), if $T_{ij}=0$ for all $i\in\{1,\ldots, k\}$,
then $T_{\bm q}=0$ for all ${\bm q}\,\in S(r,k)$ and $2\le r<k$. By conditions and \eqref{E-PW3-k},
\begin{equation*}
%\label{E-int-k-harm}
 %%%%%%%%%%%
 2\,{\rm S}_{\rm mix}({\cal D}_1,\ldots,{\cal D}_k) +\sum\nolimits_{r\in\{1,k-1\}} \sum\nolimits_{\,{\bm q}\,\in S(r,k)}(\|h_{\bm q}\|^2 =0.
\end{equation*}
Thus, $h_{\bm q}=0$ for all ${\bm q}\,\in S(r,k)$ with $r\in\{1,k-1\}$, in particular, $h_i=0\ (1\le i\le k)$.
By~well-known de Rham decomposition theorem, $(M,g)$~splits.
\end{proof}

\begin{theorem}
%\label{C-Step3}
Let a complete open Riemannian almost $k$-product manifold $(M,g;{\cal D}_1,\ldots,{\cal D}_k)$ with $k>2$ has totally umbilical distributions
such that each pair $({\cal D}_i,{\cal D}_j)$ is mixed totally geodesic and $\<H_i,H_j\>=0$ for $i\ne j$.
If $\,{\rm S}_{\rm mix}({\cal D}_1,\ldots,{\cal D}_k)\le0$ and $\|H_i\|\in{\rm L}^1(M,g)$ for $1\le i\le k$,
then $M$~splits.
\end{theorem}

\begin{proof}
 By assumptions, from \eqref{E-PW3-k} we get
\begin{equation}\label{E-int-k-umb}
 %%%%%%%%%%%
 \Div \xi = 2\,{\rm S}_{\rm mix}({\cal D}_1,\ldots,{\cal D}_k)
 +\sum\nolimits_{r\in\{1,k-1\}}\sum\nolimits_{\,{\bm q}\,\in S(r,k)}\big(\|h_{\bm q}\|^2-\|H_{\bm q}\|^2-\|T_{\bm q}\|^2\big) ,
\end{equation}
where $\xi=\!\sum\nolimits_{r\in\{1,k-1\}}\sum\nolimits_{\,{\bm q}\,\in S(r,k)} H_{\bm q}$.
By conditions, and since $\|H_{\bm q}\|\le \sum_{i=1}^k \|H_{q(i)}\|$, we get $\|H_{\bm q}\|\in{\rm L}^1(M,g)$.
Since $\|\xi\|\le \sum\nolimits_{r\in\{1,k-1\}}\sum\nolimits_{\,{\bm q}\,\in S(r,k)}\|H_{\bm q}\|$,
then also $\|\xi\|\in{\rm L}^1(M,g)$.
Next, by conditions, for any ${\bm q}=(q(1),\ldots,q(r))\in S(r,k)$ with $r\ge1$ we have
\begin{equation*}
%\label{E-int-k-umb2}
 \|h_{\bm q}\|^2-\|H_{\bm q}\|^2 = -\sum\nolimits_{i=1}^r\frac{n_{q(i)}-1}{n_{q(i)}}\,\|\widehat P_{\bm q}H_{q(i)}\|^2 \le0,
\end{equation*}
where $\widehat P_{\bm q}$ is the orthoprojector on the distribution $\bigoplus_{j\notin{\bm q}} {\cal D}_{j}$.
Hence, from ${\rm S}_{\rm mix}({\cal D}_1,\ldots,{\cal D}_k)\le0$ and \eqref{E-int-k-umb} we get $\Div\xi\le0$.
By conditions and Lemma~\ref{L-Div-1} below, $\Div\xi=0$. Thus, see \eqref{E-int-k-umb}, ${\rm S}_{\rm mix}({\cal D}_1,\ldots,{\cal D}_k)=0$ and ${T}_i$ and $h_i$ vanish.
By de Rham decomposition theorem, $(M,g)$ splits.
\end{proof}

Modifying Stokes' theorem on a complete open manifold $(M,g)$ yields the following.

\begin{lemma}[see Proposition~1 in \cite{csc2010}]
\label{L-Div-1}
 Let $(M^n,g)$ be a complete open Riemannian manifold endowed with a vector field $\xi$
 such that $\Div\xi\ge0$. If the norm $\|\xi\|_g\in{\rm L}^1(M,g)$ then $\Div\xi\equiv0$.
\end{lemma}

\begin{example}
%\label{Ex-01}
\rm
Totally umbilical integrable distributions appear on multiply warped products.
Let $F_1,\ldots,F_k$ be $k$ Riemannian manifolds and let $M=F_1\times\ldots\times F_k$ be their direct product.
A~\textit{multiply warped product} $F_1\times_{u_2}F_2\times\ldots\times_{u_k} F_k$
is $M$ with the metric $g=g_{F_1}\oplus u_2^2\,g_{F_2}\oplus\ldots\oplus u_k^2\,g_{F_k}$, where $u_i:F_1\to(0,\infty)$
for $i=2,\ldots,k$ are smooth functions.
Let ${\cal D}_i$ be the distribution on $M$ obtained from the vectors tangent to (the horizontal lifts of) $F_i$, e.g., \cite{chen1}.
The {leaves} (i.e., tangent to ${\cal D}_i,\ i\ge2$) are totally umbilical submanifolds, with
\begin{eqnarray*}
  H_i=-n_i\nabla(\log u_i),
\end{eqnarray*}
and the {fibers} (i.e., tangent to ${\cal D}_1$) are totally geodesic submanifolds.
Since
\begin{equation*}
 \Div\,{H}_i = -n_i\,(\Delta\,u_i)/u_i -(n_i^2-n_i)\,\|P_{\hat i}\nabla u_i\|^2/u_i^2,
\end{equation*}
where $\Delta=-\Div\circ\nabla$ is the Laplacian on $C^2(F_1)$, and we have
\begin{equation}\label{E-Kmix-warped}
 {\rm S}_{\rm mix}({\cal D}_1,\ldots,{\cal D}_k) = \sum\nolimits_{\,i\ge2} n_i\,(\Delta\,u_i)/u_i\,.
\end{equation}
\end{example}

\begin{corollary}
%[\rm of Theorem~\ref{C-Step3}]
Let a {multiply warped product} $(M,g)$ be complete open and $\<H_i,H_j\>=0$ for $i\ne j$.
If $\,{\rm S}_{\rm mix}({\cal D}_1,\ldots,{\cal D}_k)\le0$ and $\|H_i\|\in{\rm L}^1(M,g)$ for $1\le i\le k$,
then $M$ is the direct product.
\end{corollary}

The following theorem on isometric immersions of multiply warped products, see \cite[Theorem~10.2]{chen1},
is related to the question by B.-Y.~Chen: 
``\textit{What can we conclude from an arbitrary isometric immersion of a warped product into a Riemannian manifold with arbitrary codimension}?"

\textbf{Theorem A}
\textit{Let $f : F_1\times_{u_2}F_2\times\ldots\times_{u_k} F_k \to \tilde M$ be an isometric immersion of
a multiply warped product $(M,g) := F_1\times_{u_2}F_2\times\ldots\times_{u_k} F_k$ into an arbitrary Riemannian
manifold. Then
\begin{equation}\label{E-C-D}
 \sum\nolimits_{i=2}^k n_i\,(\Delta u_i)/u_i \le \frac{n^2(k-1)}{2\,k}\,H^2+ n_1(n-n_1)\max\tilde K,\quad
 n=\sum\nolimits_{\,j=1}^k n_j,
\end{equation}
where $H^2 = \<H,H\>$ is the squared mean curvature of $f$,
$\max\tilde K(x)$ is the maximum of the sectional curvature of $\tilde M$
restricted to 2-plane sections of the tangent space $T_xM$ of $M$ at $x=(x_1,\ldots,x_k)$.
The equality sign of \eqref{E-C-D} holds identically if and only if the following two statements hold}:

1) \textit{$f$ is a mixed totally geodesic immersion satisfying equalities $\tr_g h_1 =\ldots = \tr_g h_k$};

2) \textit{at each point $x\in M$, the sectional curvature function $\tilde K$ of $\tilde M$ satisfies
$\tilde K(X,Y) = \max\tilde K(x)$ for each unit vector $X\in T_{x_1}(F_1)$ and each unit vector $Y$ in
$T(x_2,\ldots, x_k)(F_2 \times \ldots \times F_k)$.}

\smallskip
On a multiply warped product with $k>2$ each pair of distributions is mixed totally geodesic.
Indeed, such manifold is diffeomorphic to the direct product and the Lie bracket does not depend on metric.
Thus, using Theorem A and our results above, we obtain the following.

\begin{theorem}[Non-existence of immersions]\label{T-C-D}
Let $f : F_1\times_{u_2}F_2\times\ldots\times_{u_k} F_k \to \tilde M$ be an isometric immersion of a closed multiply warped product $(M,g) := F_1\times_{u_2}F_2\times\ldots\times_{u_k} F_k$ into an arbitrary Riemannian manifold.
Let, in addition, $\<H_i,H_j\>=0$ for $i\ne j$.
Then there are no isometric immersions $f: F_1\times_{u_2}F_2\times\ldots\times_{u_k} F_k \to \tilde M$ satisfying the inequality
\begin{equation}\label{Eq-immerse}
 \frac{n^2(k-1)}{2\,k} H^2 < -n_1(n-n_1)\max\tilde K;
\end{equation}
in particular, there are no minimal isometric immersions $f$ when $\tilde K<0$.
\end{theorem}

\begin{proof}
From \eqref{E-Kmix-warped}, \eqref{E-int-k-umb} and \eqref{E-C-D}, applying the Stokes' Theorem, we get
\begin{eqnarray}\label{E-C-Dmy}
\nonumber
 && 0\le (1/2)\int_M \sum\nolimits_{r\in\{1,k-1\}}\sum\nolimits_{\,{\bm q}\in S(r,k)}
 \sum\nolimits_{i=1}^r\frac{n_{q(i)}-1}{n_{q(i)}}\,\|\widehat P_{\bm q}H_{q(i)}\|^2
 \,{\rm d}\vol_g \\
 &&\qquad \le \frac{n^2(k-1)}{2\,k}\int_M H^2\,{\rm d}\vol_g +\, n_1(n-n_1)\max\tilde K\cdot{\rm Vol}(M,g).
\end{eqnarray}
The equality in \eqref{E-C-Dmy} holds if and only if conclusions 1)--2) of Theorem A are satisfied.
If such isometric immersion exists, then the inequality \eqref{Eq-immerse} yields a contradiction
\end{proof}

\begin{remark}\rm
Theorem~\ref{T-C-D} can be improved replacing $\max\tilde K$ by $\bar S(n_1,\ldots,n_k)$ (with certain factor) defined below.
Let $L_i\ (i=1,2)$ be two subspaces of $T_x\bar M$ at some point $x\in\bar M$,
and $\{E'_i\}$ and $\{{\cal E}'_j\}$ some orthonormal frames of these subspaces.
The following quantity:
\[
 \bar S(L_1,L_2)=\sum\nolimits_{\,i,j}\<\bar R(E'_i,{\cal E}'_j)\,E'_i,\,{\cal E}'_{j}\>
\]
does not depend on the choice of frames.
For any $n_1,\ldots,n_k\in\NN$ with $\sum_{\,i}n_i\le\dim\bar M$,
let $\Gamma_x(n_1,\ldots,n_k)$ be the set of $k$-tuples of pairwise orthogonal subspaces $(L_1,\ldots,L_k)$
with $\dim L_i=n_i\ (1\le i\le k)$ of $T_x\bar M$ at some point $x\in\bar M$.
Set $\bar S(n_1,\ldots,n_k)=\sup\nolimits_{\, x\in\bar M}\bar S_x(n_1,\ldots,n_k)$, where
\[
 \bar S_x(n_1,\ldots,n_k)=\max\{\sum\nolimits_{\,i<j}\bar S(L_i,L_j):\,(L_1,\ldots,L_k)\in\Gamma_x(n_1,\ldots,n_k)\}.
\]
\end{remark}

\section{Hypersurfaces with $k$ distinct principal curvatures}
\label{sec:dupin}

Let $M^n$ be a transversely orientable hypersurface in a Riemannian manifold $\bar M$.
Denote by $A$ the shape operator of $M$ with respect to a unit normal vector field $N$,
and let $\lambda_1\le\ldots\le\lambda_n$ be the principal curvatures (eigenvalues of $A$) -- continuous functions on~$M$.
If at any point of $M$ all principal curvatures are equal, then $M$ is totally umbilical.

Suppose that there exists $k\ge2$ distinct principal curvatures $(\mu_i)$ on $M$ of multiplicities $n_i$, and let ${\cal D}_i\ (1\le i\le k)$ be corresponding eigen-distributions -- subbundles of $TM$.
\[
 \mu_1:=\lambda_1=\ldots=\lambda_{n_1} < \ \ldots \ < \mu_k:=\lambda_{n-n_{k}+1}=\ldots=\lambda_{n}.
\]
Such hypersurface is a Riemannian almost $k$-product manifold $(M,g;{\cal D}_1,\ldots,{\cal D}_k)$, see Section~\ref{sec:if}.

For a hypersurface $M$ in a Riemannian manifold $\bar M(c)$ of constant curvature $c$, it is known the following, see \cite{cecil-ryan}:
if $n_i>1$ for some $i$, then the function $\mu_i:M\to\RR$ is differentiable and ${\cal D}_i$ is integrable and its leaves are totally umbilical in $\bar M(c)$.

The curvature tensor $\bar R$ of $\bar M(c)$ has the well-known view
\[
 \bar R(X,Y)Z = c\,(\<X,Z\>\,Y - \<Y,Z\>\,X).
\]
Denote by $X_i\in\mathfrak{X}_i$ local unit vector fields on ${\cal D}_i\ (i\le k)$.
For the sectional curvature of $M$ we~get
\begin{equation}\label{E-Kmix-ij}
 K(X_i,X_j)=c+\mu_i\,\mu_j,\quad i\ne j.
\end{equation}

\begin{example}\label{Ex-06-PW}\rm For a hypersurface $M$ in $\bar M(c)$ with $k=2$,
using equalities
\[
 \|h_i\|^2-\|H_i\|^2=n_i(1-n_i)\,\frac{\|\nabla\mu_i\|^2}{(\mu_i-\mu_j)^2},\quad
 {\rm S}_{\rm mix}({\cal D}_1,{\cal D}_2)=n_1n_2(c+\mu_1\mu_2),
\]
formula \eqref{E-PW} can be rewritten in terms of $A$ and $\mu_i\ (i=1,2)$ as follows, see \cite{wa1}:
\begin{equation}\label{E-PW-M}
 \Div(H_1 + H_2) = n_1n_2(c+\mu_1\mu_2)
 +\frac{n_1(1-n_1)\,\|\nabla\mu_1\|^2 +n_2(1-n_2)\,\|\nabla\mu_2\|^2}{(\mu_2-\mu_1)^2},
\end{equation}
\end{example}

The next \textbf{problem} was posed by P.~Walczak \cite{wa1}:
``To search for formulae analogous to~\eqref{E-PW-M} in the case of a hypersurface in $\bar M(c)$
of $k>2$ distinct principal curvatures of constant multiplicities".

\smallskip

\noindent
By symmetry of $A$ and the Codazzi equation for a hypersurface in $\bar M(c)$,
\begin{equation}\label{E-Cor-hyper}
 (\nabla_{X}A)(Y)=(\nabla_{Y}A)(X),
\end{equation}
the following tensor ${\cal A}$ of rank 3 is totally symmetric:
\[
 {\cal A}(X,Y,Z)=\<(\nabla_{X}A)(Y),\,Z\>, \quad X,Y,Z\in\mathfrak{X}_M.
\]

\begin{remark}\rm
A submanifold $M$ in $\bar M$ is called \textit{curvature-invariant} if $R(X,Y)Z\in T_xM$ for any $X,Y,Z\in T_xM$ and any $x\in M$, e.g.,~\cite{rov-m}.
Examples are arbitrary submanifolds in $\bar M(c)$.
For a curvature-invariant hypersurface $M\subset\bar M$, Codazzi equation \eqref{E-Cor-hyper} is satisfied,
thus, results of this section on a hypersurface in $\bar M(c)$ can be extended for curvature-invariant hypersurfaces.
\end{remark}

\begin{lemma} For a hypersurface $M$ in $\bar M(c)$ with $k>2$ distinct principal curvatures, we have
\begin{eqnarray}\label{E-ijk-1}
\nonumber
 {\cal A}(X_i,X_j,X_l) \eq (\mu_j-\mu_l)\<\nabla_{X_i}X_j,\,X_l\>,\quad j\ne l,\\
 {\cal A}(X_i,X_j,X_j) \eq X_i(\mu_j),\quad \forall\ i, j.
\end{eqnarray}
\end{lemma}

\begin{proof} After derivation of equality $AX_j=\mu_j X_j$ and using symmetry of $A$, we obtain
\begin{eqnarray*}
 \<\nabla_{X_i}(AX_j),\,X_l\> \eq \<(\nabla_{X_i} A)(X_j),\,X_l\> + \<A\nabla_{X_i}X_j,\,X_l\> \\
  \eq {\cal A}(X_i,X_j,X_l) + \mu_l\<\nabla_{X_i}X_j,\,X_l\> ,\\
 \<\nabla_{X_i}(\mu_j X_j),\,X_l\> \eq \mu_j\,\<\nabla_{X_i}X_j,\,X_l\> + X_i(\mu_j)\,\delta_{jl}.
\end{eqnarray*}
From this equalities in \eqref{E-ijk-1} follow.
\end{proof}

\begin{corollary} For a hypersurface $M$ in $\bar M(c)$ with $k>2$ distinct principal curvatures, we have
\begin{equation}\label{E-ijk-0}
 (\mu_j-\mu_l)\,\<\nabla_{X_i}X_j,\,X_l\> = (\mu_i-\mu_l)\,\<\nabla_{X_j}X_i,\,X_l\>,\quad i\ne j\ne l.
\end{equation}
\end{corollary}

\begin{proof} This follows from \eqref{E-ijk-1} and symmetry of ${\cal A}$.
Alternatively, one can use direct derivation,
\begin{eqnarray*}
 0 = \<\bar R(X_i,X_j)N,\,X_l\> \eq \<\mu_{X_i}(\mu_j X_j) - \nabla_{X_j}(\mu_i X_i)
 -\<[X_i,X_j],\,X_l\>\<\mu_l X_l,\, X_l\>\\
 \eq (\mu_j-\mu_l)\,\<\nabla_{X_i}X_j,\,X_l\> -(\mu_i-\mu_l)\,\<\nabla_{X_j}X_i,\,X_l\>,
\end{eqnarray*}
from which \eqref{E-ijk-0} follows.
\end{proof}

Theorem~5.13 in \cite{cecil-ryan}, which states that
\textit{each point of a hypersurface $M^n$ in $\bar M(c)$ with 2, 3, 4 or 6 distinct principal curvatures has a principal coordinate neighborhood if and only if each ${\cal D}_i^\bot$ is integrable}. We complete this result by the following.

\begin{proposition}
Let $M^n$ be a hypersurface in $\bar M(c)$ with $k>2$ distinct principal curvatures at each point.
Then each ${\cal D}_i^\bot$ is integrable if and only if ${\cal A}(X_i,X_j,X_l)=0$ for $X_i\in{\cal D}_i,\,X_j\in{\cal D}_j$ and $X_l\in{\cal D}_l$ with $i<j<l$ on $M^n$.
\end{proposition}

Next, we study the problem (posed in \cite{wa1}) for $k=3$ (the case of $k>3$ is similar).

\begin{theorem}\label{T-dupin}
For a hypersurface $M$ in $\bar M(c)$ with $k=3$ distinct principal curvatures, we have
\begin{eqnarray}\label{E-k-3-surf}
\nonumber
 &&\hskip-5mm \Div\sum\nolimits_{\,i}
    n_i\Big(\frac{P_j\nabla\mu_i}{\mu_i-\mu_j} +\frac{P_l\nabla\mu_i}{\mu_i-\mu_l} \Big) \\
 && = \frac12\sum\nolimits_{\,i<j} n_i\,n_j(c+\mu_i\mu_j)
  +\sum\nolimits_{\,i} n_i(1-n_i)\Big(\frac{\|P_j\nabla\mu_i\|^2}{(\mu_i-\mu_j)^2} +\frac{\|P_l\nabla\mu_i\|^2}{(\mu_i-\mu_l)^2} \Big),
\end{eqnarray}
where $(j,l)\in \{1,2,3\}\setminus\{i\}$ and $j<l$.
\end{theorem}

\begin{proof}
By \eqref{E-PW-3} and $T_i=0$, we have
\begin{eqnarray}\label{E-k3-surf-1}
\nonumber
 \Div\big(\sum\nolimits_{\,i} H_i +\sum\nolimits_{\,i<j} H_{ij}\big) \eq 2\,{\rm S}_{\rm mix}({\cal D}_1,{\cal D}_2,{\cal D}_3)
 +\sum\nolimits_{\,i}(\|h_i\|^2 -\|H_i\|^2) \\
 \plus\sum\nolimits_{\,i<j}(\|h_{ij}\|^2 - \|H_{ij}\|^2 -\|T_{ij}\|^2) .
\end{eqnarray}
By \eqref{E-Kmix-ij}, the mixed scalar curvature of $M$ with $k=3$ is
\[
 {\rm S}_{\rm mix}({\cal D}_1,{\cal D}_2,{\cal D}_3)=\sum\nolimits_{\,i<j} n_i\,n_j(c+\mu_i\mu_j).
\]
Note that $\|h_i\|^2=\|H_i\|^2/n_i$, where
\begin{equation*}
%\label{E-k3-surf-2}
 h_i(X,Y) = \<X,Y\>\Big(\frac{P_j\nabla\mu_i}{\mu_i-\mu_j} +\frac{P_l\nabla\mu_i}{\mu_i-\mu_l}\Big),\quad
 H_i = n_i\Big(\frac{P_j\nabla\mu_i}{\mu_i-\mu_j} +\frac{P_l\nabla\mu_i}{\mu_i-\mu_l}\Big).
\end{equation*}
Thus (similarly to Example~\ref{Ex-06-PW}), we obtain
\begin{eqnarray*}
%\label{E-k3-surf-3}
 \|h_i\|^2 -\|H_i\|^2 = n_i(1-n_i)\Big(\,\frac{\|P_j\nabla\mu_i\|^2}{(\mu_i-\mu_j)^2} +\frac{\|P_l\nabla\mu_i\|^2}{(\mu_i-\mu_l)^2}\,\Big).
\end{eqnarray*}
Next, we consider $\|h_{ij}\|^2$ and $\|T_{ij}\|^2$ for $i\ne j$ along three subsets of $TM\oplus TM$:
${\cal D}_{i}\oplus{\cal D}_i$, ${\cal D}_j\oplus{\cal D}_j$ and $V_{ij}=({\cal D}_i\oplus{\cal D}_j)\cup({\cal D}_j\oplus{\cal D}_i)$.
Using
 $h_{ij} = \frac{n_i\,P_l\nabla\mu_i}{\mu_i-\mu_l} + \frac{n_j\,P_l\nabla\mu_j}{\mu_j-\mu_l}$ for $i\ne j\ne l$,
we obtain
\begin{eqnarray*}
%\label{E-k3-surf-3-12}
 \|h_{ij}\|^2 -\|H_{ij}\|^2 = n_i(1-n_i)\frac{\|P_l\nabla\mu_i\|^2}{(\mu_i-\mu_l)^2}
 +n_j(1-n_j)\frac{\|P_l\nabla\mu_j\|^2}{(\mu_j-\mu_l)^2} + \|h_{ij}|_{\,V_{ij}}\|^2,
\end{eqnarray*}
Since ${\cal D}_i$ are integrable, we have $\|T_{ij}\|^2=\|T_{ij}|_{\,V_{ij}}\|^2$.
Given $i\ne j\ne l$, let $(e_\alpha),\,(\varepsilon_\beta),\,(E_\gamma)$ be local orthonormal frames of
${\cal D}_i,{\cal D}_j$ and ${\cal D}_l$, respectively.
Using \eqref{E-ijk-1}, and
\[
 \|h_{ij}|_{\,V_{ij}}\|^2 = \frac14\sum\nolimits_{\,\alpha,\beta,\gamma}\<\nabla_{e_\alpha}\,\varepsilon_\beta+\nabla_{\varepsilon_\beta}\,e_\alpha,\,E_\gamma\>^2,\quad
 \|T_{ij}|_{\,V_{ij}}\|^2 = \frac14\sum\nolimits_{\,\alpha,\beta,\gamma}\<\nabla_{e_\alpha}\,\varepsilon_\beta-\nabla_{\varepsilon_\beta}\,e_\alpha,\,E_\gamma\>^2,
\]
we find for $(i,j)\in\{(1,2),\,(1,3),\,(2,3)\}$ the equality
\[
 \|h_{ij}|_{\,V_{ij}}\|^2 - \|T_{ij}|_{\,V_{ij}}\|^2
 =\sum\nolimits_{\,\alpha,\beta,\gamma}\<\nabla_{e_\alpha}\,\varepsilon_\beta,\,E_\gamma\>\,\<\nabla_{\varepsilon_\beta}\,e_\alpha,\,E_\gamma\> \\
 =\frac{\sum\nolimits_{\,\alpha,\beta,\gamma}\|{\cal A}(e_\alpha,\,\varepsilon_\beta,\,E_\gamma)\|^2} {(\mu_i-\mu_l)(\mu_j-\mu_l)}\,.
\]
Since
\[
  \frac{1}{(\mu_2-\mu_3)(\mu_1-\mu_3)}
 +\frac{1}{(\mu_1-\mu_2)(\mu_3-\mu_2)}
 +\frac{1}{(\mu_2-\mu_1)(\mu_3-\mu_1)}=0,
\]
the sum
$\sum\nolimits_{\,i<j} \big(\|h_{ij}|_{\,V_{ij}}\|^2 - \|T_{ij}|_{\,V_{ij}}\|^2\big)$
vanishes.  From the above, we have
\begin{eqnarray*}
 &&\hskip-7mm \sum\nolimits_{\,i}(\|h_i\|^2 {-}\|H_i\|^2) + \!\sum\nolimits_{\,i<j}(\|h_{ij}\|^2 {-} \|H_{ij}\|^2 {-}\|T_{ij}\|^2)
 = 2 n_1(1{-}n_1)\Big(\frac{\|P_2\nabla\mu_1\|^2}{(\mu_1-\mu_2)^2} {+}\frac{\|P_3\nabla\mu_1\|^2}{(\mu_1-\mu_3)^2} \Big)
 \\ && +\,2\, n_2(1-n_2)\Big(\frac{\|P_1\nabla\mu_2\|^2}{(\mu_2-\mu_1)^2} +\frac{\|P_3\nabla\mu_2\|^2}{(\mu_2-\mu_3)^2} \Big)
  + 2\, n_3(1-n_3)\Big(\frac{\|P_1\nabla\mu_3\|^2}{(\mu_3-\mu_1)^2} +\frac{\|P_2\nabla\mu_3\|^2}{(\mu_3-\mu_2)^2} \Big).
\end{eqnarray*}
Note that the terms under the divergence in \eqref{E-k3-surf-1} are
\begin{eqnarray*}
%\label{E-k-3-surf-long}
 && \sum\nolimits_{\,i} H_i {+}\sum\nolimits_{\,i<j} H_{ij} =
 \sum\nolimits_{i} n_i\Big(\frac{P_j\nabla\mu_i}{\mu_i-\mu_j} +\frac{P_l\nabla\mu_i}{\mu_i-\mu_l}\Big)
 {+}\sum\nolimits_{i<j} \Big(n_i\,\frac{P_l\nabla\mu_i}{\mu_i-\mu_l} + n_j\,\frac{P_l\nabla\mu_j}{\mu_j-\mu_l}\Big) \\
%%%%
 && =  2\,n_1\Big(\frac{P_2\nabla\mu_1}{\mu_1-\mu_2} +\frac{P_3\nabla\mu_1}{\mu_1-\mu_3} \Big)
 +2\,n_2\Big(\frac{P_1\nabla\mu_2}{\mu_2-\mu_1} +\frac{P_3\nabla\mu_2}{\mu_2-\mu_3} \Big)
 +2\,n_3\Big(\frac{P_1\nabla\mu_3}{\mu_3-\mu_1} +\frac{P_2\nabla\mu_3}{\mu_3-\mu_2} \Big).
\end{eqnarray*}
From the above \eqref{E-k-3-surf} follows.
\end{proof}

Using Stokes' theorem to \eqref{E-k-3-surf} on a closed $M$, we obtain the following integral formula.

\begin{corollary} In conditions of Theorem~\ref{T-dupin}, for a closed hypersurface $M\subset\bar M(c)$, we have
\begin{eqnarray*}
%\label{E-k-3-surf}
 && \int_M\Big[\sum\nolimits_{\,i<j} n_i\,n_j(c+\mu_i\mu_j)
 %+  n_1(1-n_1)\Big(\frac{\|P_2\nabla\mu_1\|^2}{(\mu_1-\mu_2)^2}
  +2\sum\nolimits_{\,i} n_i(1-n_i)\Big(\frac{\|P_j\nabla\mu_i\|^2}{(\mu_i-\mu_j)^2} +\frac{\|P_l\nabla\mu_i\|^2}{(\mu_i-\mu_l)^2} \Big) \Big]\,{\rm d}\vol_g =0,
\end{eqnarray*}
where $(j,l)\in \{1,2,3\}\setminus\{i\}$ and $j<l$.
 \end{corollary}


\begin{thebibliography}{999.}%

\bibitem{AG2000}
 M.A. Akivis, V.V. Goldberg, Differential geometry of webs. In \textit{Handbook of differential geometry}, Vol. I, pp. 1--152, North-Holland, Amsterdam, 2000

\bibitem{arw2014}
 K. Andrzejewski, V. Rovenski and P. Walczak, {\em Integral formulas in foliation theory}, 73--82,
in ``\textit{Geometry and its Applications}", Springer Proc. in Math. and Statistics, 72, Springer,~2014

\bibitem{bf}
A.~Bejancu and H.~Farran.  \textit{Foliations and geometric structures}, Dordrecht: Springer, 2006

\bibitem{BM-1}
M. Banaszczyk, R. Majchrzak, An integral formula for a Riemannian manifold with three orthogonal distributions.
Acta Sci. Math., 54 (1990), 201--207

\bibitem{csc2010}
A. Caminha, P. Souza, F. Camargo, Complete foliations of space forms by hypersurfaces,
Bull. Braz. Math. Soc., New Series, {41}:3 (2010), 339--353

\bibitem{cecil-ryan}
T.E. Cecil, P.J. Ryan,  \textit{Geometry of Hypersurfaces}, Springer Monographs in Mathematics. Springer-Verlag, New York, 2015

\bibitem{chen1}
B.-Y.~Chen, \textit{Differential geometry of warped product manifolds and submanifolds}. World Scientific, 2017

\bibitem{dug}
K.L. Duggal, B. Sahin, \textit{Differential Geometry of Lightlike Submanifolds}, Birkh\"{a}user, 2010

\bibitem{g1967}
A. Gray, {Pseudo-Riemannian almost product manifolds and submersions}, J. Math. Mech. {16}:7 (1967), 715--737

\bibitem{ots-86}
T. Otsuki, Minimal hypersurfaces with three principal curvatures in $S^{n+1}$. Kodai Math. J. 1 (1978), 1--29

\bibitem{ra1}
A. Ranjan, Structural equations and an integral formula for foliated manifolds, Geom. Dedicata, 20 (1986), 85--91

\bibitem{reeb1}
G. Reeb, {Sur la courboure moyenne des vari\'et\'es int\'egrales d'une \'equation de Pfaff $\omega = 0$},
C. R. Acad. Sci. Paris {231} (1950), 101--102

\bibitem{rov-m}
V. Rovenski, \textit{Foliations on Riemannian Manifolds and Submanifolds}. Birkh{\"a}user, 1998

\bibitem{step1}
S.E. Stepanov and J. Mike\v{s},
Liouvile-type theorems for some classes of Riemannian almost pro\-duct ma\-nifolds and for special mappings of Riemannian manifolds,
Differential Geom. and its Appl. 54, Part A (2017), 111--121

\bibitem{tar}
A. Tarrio, On certain clacces of metric para-$\phi$-manifolds with parallelizable kernel,
Tebsir, N.S. 57 (1996), 258--267

\bibitem{tol}
G.A. Tolstikhina,  Algebra and geometry of 3-webs formed by foliations of different dimensions.
%(Russian) Sovrem. Mat. Prilozh. No. 32, Geom. Geom. Tkan. (2005), 29–115; translation in
J. Math. Sci. (N.Y.) 143 (2007), no. 6, 3630--3721

\bibitem{wa1}
P. Walczak, {An integral formula for a Riemannian manifold with two orthogonal complementary distributions}.
Colloq. Math., {58} (1990), 243--252

\end{thebibliography}
\end{document}